\documentclass[a4paper,12pt]{article}
\usepackage{graphicx}
\usepackage{amssymb}
\usepackage[utf8]{inputenc}
\usepackage{fancyhdr}		
\usepackage{amsmath}
\usepackage{verbatim} 
\usepackage{amsthm}

\usepackage{aliascnt}  
\usepackage{hyperref}

\newaliascnt{lem}{satz}  
\newtheorem{lem}[lem]{Lemma}  
\aliascntresetthe{lem}

\newaliascnt{thm}{satz}  
\newtheorem{thm}[thm]{Theorem}  
  \aliascntresetthe{thm}

\newaliascnt{clm}{satz}  
  
  \aliascntresetthe{clm}

\newaliascnt{cor}{satz}  
\newtheorem{cor}[cor]{Corollary}  
  \aliascntresetthe{cor}

\newaliascnt{conj}{satz}  
  
  \aliascntresetthe{conj}

\newaliascnt{prop}{satz}  
\newtheorem{prop}[prop]{Proposition}  
  \aliascntresetthe{prop}

\newpage

\begin{document}

\newcommand{\h}{non-constant harmonic function of finite energy }
\newcommand{\hdot}{non-constant harmonic function of finite energy}
\newcommand{\frho}{a $p$-$q$-flow $f$ of finite energy with intensity $I>0$ and a potential $\rho$ of finite energy in $G/f$ with $\rho(p)\neq \rho(q)$ }
\newcommand{\frhodot}{a $p$-$q$-flow $f$ of finite energy with intensity $I>0$  and a potential $\rho$ of finite energy in $G/f$ with $\rho(p)\neq \rho(q)$}

\title{A characterization of the locally finite networks admitting non-constant harmonic functions of finite energy}
\author{Johannes Carmesin}
\newpage
\maketitle

\abstract{
We characterize the locally finite networks admitting non-constant harmonic functions of finite energy.
Our characterization unifies the necessary existence criteria of Thomassen \cite{thomassenCurrents2,thomassenCurrents} and of Lyons and Peres \cite{LyonsBook}
with the sufficient criterion of Soardi \cite{SoardiBook}.

We also extend a necessary existence criterion for non-elusive non-constant harmonic functions of finite energy 
due to Georgakopoulos~\cite{agelos_unique_flow}.
}

\section{Introduction}

One of the standard problems
in the study of infinite electrical networks is to specify under what conditions
a network is in  ${\cal O}_{HD}$, that is, every harmonic function of finite energy is constant \cite{{LyonsBook}, {SoardiBook}, {Uni_So_Wo},{thomassenCurrents}}. 
The purpose of this paper is to characterize the networks in  ${\cal O}_{HD}$.

\medskip

There are two general sufficient criteria for a network to be in ${\cal O}_{HD}$.
Let us illustrate these by a simple example, the infinite ladder shown in \autoref{ladder}.

    \begin{figure}[htp]
\begin{center}
    	  \includegraphics[width=5cm, height= 2cm]{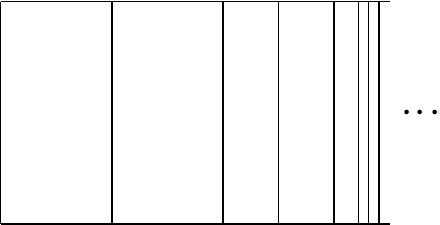}
   	  \caption{For which resistance function is the infinite ladder in ${\cal O}_{HD}$?}
   	  \label{ladder}
\end{center}
   \end{figure}

The first criterion, due to Thomassen \cite{thomassenCurrents2} and to Lyons and Peres \cite{LyonsBook}, 
implies that this network is in  ${\cal O}_{HD}$ if the resistances of the rungs are small enough,
the sum of their conductances is infinite.
The second, folklore, criterion \cite{LyonsBook} is that a network is in ${\cal O}_{HD}$ 
if it is recurrent. For the ladder, Nash-Williams's recurrence criterion \cite{LyonsBook} implies that
this is the case if on each side of the ladder the sum of the resistances is infinite.

\medskip

Our characterization of the networks in ${\cal O}_{HD}$ implies both these sufficient criteria. Conversely it shows that, in a sense,  
they are the only two reasons that can force a network to be in ${\cal O}_{HD}$.
Let $G/A/B$ be the graph obtained from $G$ by contracting each of the disjoint sets $A$ and $B$ to a vertex.
Our characterization is:

\begin{thm}
\label{char}
 A connected locally finite network $(G,r)$ is not in ${\cal O}_{HD}$ if and only if 
there are  transient vertex-disjoint subnetworks $A$ and $B$ such that the contraction $G/A/B$ admits a potential $\rho$ of finite energy with $\rho(A)\neq\rho(B)$.
\end{thm}

Since networks containing transient networks are transient, it is clear that \autoref{char} implies the second sufficient criterion mentioned earlier.
It is also not hard to deduce the sufficient criterion of Thomassen, Lyons and Peres formally from \autoref{char}; see Section~\ref{sec:cor}.

In our ladder example, it is easy to show that up to slight modification the
only two
transient vertex-disjoint subnetworks $A$ and $B$ of the infinite ladder 
are the two infinite sides of the ladder.
It is easy to show that a side of the ladder is transient if and only if the sum over its resistances is finite.
As here $G/A/B$ has only the two contraction vertices $A$ and $B$, the unique (up to adding a constant) potential in $G/A/B$ with $\rho(A)-\rho(B)=U$ has the energy $U^2$ times the sum over the conductances of the rungs.
Hence \autoref{char} yields that the infinite ladder is in ${\cal O}_{HD}$ if and only if the sum over the conductances of the rungs is infinite or 
the sum over the resistances of any side of the ladder is infinite.
Note that the last requirement is slightly stronger than the second sufficient criterion.

\medskip

\autoref{char} also implies some new and easily applicable existence criteria for non-constant harmonic functions.
The following corollary strengthens the well-known fact \cite{agelos_unique_flow} that a network $(G,r)$ with
$\sum_{e\in E(G)} 1/r(e)<\infty$ is in ${\cal O}_{HD}$:

\begin{cor}\label{G-S}
Let $(G,r)$ be a connected locally finite network, and let $S$ be a set of edges 
such that $G-S$ is connected and $\sum_{e\in S} 1/r(e)$ is finite.
The network $(G-S,r)$ is in ${\cal O}_{HD}$ if and only if $(G,r)$ is.
(Here $G-S$ denotes the graph obtained from $G$ by deleting $S$ and then all isolated vertices.)
\end{cor}

We show that the condition `` $\sum_{e\in S} 1/r(e)$ is finite'' is best possible in a very strong sense; see Section \ref{sec:cor} for details.

Our next corollary offers an example application of \autoref{char}
where $A$, $B$ and $\rho$ can be constructed explicitly from the
properties of the graph. Its special case of unit resistances was already treated in~\cite{SoardiBook}.

\begin{cor}\label{r_fin}
 Let $(G,r)$ be a connected locally finite network.
If $G$ has a cut $F$ such that  $\sum_{e\in F} 1/r(e)$ is finite, and there are
two components of $G-F$ each containing a transient network,
then $(G,r)$ is not in ${\cal O}_{HD}$.
\end{cor}

A harmonic function is \emph{non-elusive} if it satisfies the mean-value property not only
 at vertices but, more generally, at every finite cut; see Section \ref{sec:geo} for a precise definition.
We generalize the above mentioned criterion of Thomassen, Lyons and Peres so as to extend a necessary criterion for the existence of 
non-elusive non-constant harmonic functions of finite energy due to Georgakopoulos \cite{ agelos_unique_flow},
 which needs a completely new proof.

This paper is organized as follows:
We begin in Section \ref{sec:def} by giving the basic definitions.
After proving the main result in Section \ref{sec:chara}, we draw further
conclusions from it in Section \ref{sec:cor}.
In Sections \ref{sec:geo} and \ref{sec:barri_proof}, we extend a theorem of Georgakopoulos as indicated above.

\section{Definitions and basic facts} \label{sec:def}

We will be using the terminology of Diestel \cite{DiestelBook05} for graph
theoretical terms.
All graphs will be locally finite if we do not explicitly say something
different.

A \emph{network} is a pair $(G,r)$, where $G$ is an (undirected) (multi-) graph
and 
$r:E(G)\rightarrow \mathbb{R}_{> 0}$ a function 
assigning a \emph{resistance} to every edge. 
Let $c(e):=1/r(e)$ be the conductance of $e$.
\emph{A  network is locally finite}
if the graph is.
A function $h:V(G)\to \mathbb{R}$ is called a \emph{potential}.

A \emph{harmonic function} is a potential satisfying the
\emph{mean-value property} at every vertex $v$, that is,
 $h(v)$ is the mean-value over the $h$-values of its neighbors weighted with the corresponding conductance:

\[
   h(v)= \left(  \sum_{e=\{v,w\}} c(e) \right)^{-1}    \sum_{e=\{v,w\}}   h(w) c(e) 
\]

A network is in ${\cal O}_{HD}$ if every harmonic function of finite energy is constant.

\subsection{Kirchhoff's cycle law (K2)}

A directed edge is an ordered triple $(e,x,y)$, where $e\in E(G), x,y\in e,
x\neq y$.
 For $\vec{e}=(e,x,y)$, define
 $init(\vec{e}):=x, ter(\vec{e}):=y$ and $\overleftarrow{e}:=(e,y,x)$.
Let $\vec{E}(G)$ be the set of all directed edges of $G$.

A potential $\rho$ induces a function on the directed edges via 
$f(\vec{e}):= [\rho(init(\vec{e}))-\rho(ter(\vec{e}))]/r(e)$.
This function $f$ is \emph{antisymmetric}, that is,
$f(\vec{e})=-f(\overleftarrow{e})$ holds for every directed edge $\vec{e}$.
Moreover, $f$ satisfies Kirchhoff's cycle-law, which we state after a few
definitions.
Every cycle $C$ of $G$ corresponds to a \emph{directed cycle} $\vec{C}$, defined
as follows:
Let $v_0e_0v_1...v_ne_nv_0$ be one of the orientations of $C$.
We define: $\vec{C}:=\{ (e_i,v_{i-1},v_{i})|0\leq i\leq n \}$, where $i-1$ is evaluated in $\mathbb{Z}/n\mathbb{Z}$.
Note that $\vec{C}$ does depend on the chosen orientation.
Similarly one defines for a walk $K$ a \emph{directed walk} $\vec{K}$. 

An antisymmetric function $\varphi$ on the directed edges \emph{satisfies
Kirchhoff's cycle law (K2)} if for every directed cycle
$\vec{C}$ in $G$, there holds:

\begin{equation}\label{k2}
  \sum_{\vec{e}\in \vec{C}} r(e) \varphi(\vec{e})=0 
\tag{\ensuremath{K2}}
\end{equation}

Notice that (K2) also holds for directed closed walks if it holds for all cycles.
The product $r(e) \varphi(\vec{e})$ is called the  \emph{voltage of $\vec{e}$}.
Kirchhoff's cycle law says that the sum of voltages along every cycle is zero.

\subsection{Kirchhoff's node law (K1)}

An antisymmetric function $\varphi:\vec{E}\to \mathbb{R}$ \emph{satisfies Kirchhoff's node
law (K1) at $v$} if:

\begin{equation}\label{k1}
 \sum_{\vec{e}\in \vec{E}| v=ter(\vec{e})} \varphi(\vec{e})=0 
\tag{\ensuremath{K1}}
\end{equation}

Call the sum on the left the \emph{accumulation of $\varphi$ at $v$}.
Note that a potential satisfies the mean-value property at $v$ if and only if
the induced function on $\vec{E}$ satisfies (K1) at $v$.
A \emph{$v$-flow of intensity $I$} is an antisymmetric function having  accumulation $I$ at $v$ and satisfying $(K1)$ at every other vertex.
Similarly, a \emph{$p$-$q$-flow of intensity $I$} has accumulation $-I$ at $p$ and $I$ at $q$ and satisfies $(K1)$ at every other vertex.

The following lemma implies the well-known fact that every finite connected 
network is in  ${\cal O}_{HD}$.

\begin{lem} \label{kreis}
Let $(G',r')$ be a finite network and let $f$ be a flow of intensity zero with
$f(\vec{e})>0$ for some directed edge $\vec{e}$.
Then there exists a directed cycle $\vec{C}$ with $\vec{e}\in \vec{C}$ and
$f(\vec{c})>0$ for every $\vec{c}\in\vec{C}$.
\end{lem}

\begin{proof}

Color a vertex $v$ gray if there is a directed path $\vec{P}$ from $ter(\vec{e})$ to $v$ such that $f(\vec{p})>0$ for all $\vec{p}\in \vec{P}$.
For every directed edge $\vec{g}$ pointing from a gray vertex to one that is
not gray, we have $f(\vec{g})\leq 0$.
As $f$ satisfies (K1) at the finite cut between the gray vertices and the rest, $f(\vec{g})= 0$.
Thus $\vec{e}$ is not in this cut and, since $ter(\vec{e})$ is gray, $init(\vec{e})$ is gray, too.
Thus there is a directed path $\vec{P}$ from
$ter(\vec{e})$ to $init(\vec{e})$ and this path combined with $\vec{e}$ forms the desired cycle.
\end{proof}

\subsection{Energy}

The \emph{Energy of $\varphi$} is defined as ${\cal E}(\varphi):=\sum_{e\in E}r(e)
\varphi^2(e)$.
As common in the literature
\cite{{LyonsBook},{SoardiBook}}, 
we  will only study functions of finite energy.
The requirement of finite energy turns the antisymmetric functions into a
Hilbert-space via $\langle f, g\rangle:=\sum_{e\in E}
r(e)f(\vec{e})g(\vec{e})$. 
In this Hilbert-space the norm is the square-root of the energy.
This is structurally interesting and allows us to profit from the following tool:

\begin{lem}[\cite{Ana_book}, Theorem 4.10] \label{hilb}
If $C$ is a non-empty, closed, convex subset
of a Hilbert space, then there is a unique point $y\in C$ of minimum norm among
all elements of C.
\end{lem}

Another tool is the Cauchy-Schwartz-inequality which will be used to estimate the energy of a flow.

\subsection{The free and the wired current}

Given a $p$-$q$-flow $f$, let $I(f)$ denote its intensity.
If $f$ is induced by a potential $\rho$, then the \emph{potential
difference $U(f)$ between $p$ and $q$} is $\rho(p)-\rho(q)$.
There are two special flows called the free and the wired current.
For a detailed description see \cite{LyonsBook} or \cite{p-energy} where we generalize results of this paper to functions with finite $\ell^p$-norm.

The \emph{wired current ${\cal W}[G,r,p,q,I,U]$  between $p$ and $q$ with
intensity $I$} is the unique $p$-$q$-flow with intensity $I$ and minimal energy
in $(G,r)$. 
In fact the wired current also satisfies Kirchhoff's cycle law. The parameter
$U$ is the potential difference between $p$ and $q$ depending linearly on $I$.
The ratio $R_W:=\frac{U}{I}$ is called the \emph{wired
effective resistance between $p$ and $q$}. 

The \emph{free current  $  {\cal F}[G,r,p,q,I,U]$ between $p$ and $q$ with
voltage $U$} is induced by the unique 
potential with potential difference $U$ between $p$ and $q$ and minimal energy
in $(G,r)$.
In fact the free current is also a $p$-$q$-flow of intensity $I$ depending linearly on $U$.
The ratio $R_F:=\frac{U}{I}$ is called the \emph{free effective
resistance between $p$ and $q$}.
If it is clear by the context, we will omit some of the information $G,r,p,q,I,U$.

The wired and the free current are extremal in the following sense:

\begin{thm}
 [Doyle \cite{p-energy}, or \cite{LyonsBook}]
A connected locally finite network is in ${\cal O}_{HD}$
 if and only if 
${\cal F}[p,q,I]={\cal W}[p,q,I]$ for all vertices $p$ and $q$.
\end{thm}

In the following, we will describe the free current as a limit of flows in
finite networks.
Having fixed an enumeration of the vertices, let $G[V_n]$ be the subgraph of $G$ induced
on the first $n$ vertices.
Note that we can force every $G[V_n]$ to be connected and assume that $n$ is so
big that $p,q\in G[V_n]$.
Fixing $U>0$, let $F_n$ be the unique $p$-$q$-flow in the finite network on
$G[V_n]$ with potential difference $U$.

It can be shown that $lim_{n\rightarrow \infty} F_n(\vec{e})= {\cal 
F}[U](\vec{e})$ for every edge $e$ and $lim_{n\rightarrow \infty} {\cal E}(F_n)=
{\cal E}({\cal  F}[U])$.
As $F_n$ is a $p$-$q$-flow in a finite network, there holds ${\cal
E}(F_n)=I(F_n) U$, see for example \cite{Biggs} (Proposition 18.1.). This
yields:

\begin{equation}\label{uri}
 {\cal E}({\cal F}[I,U])=I U \text{ and }  {\cal E}({\cal F}[I,U])=U^2/R_F     
\end{equation}

There is a similar description for the
wired current as a limit of flows in finite networks , see \cite{LyonsBook} (Proposition 9.2.).
As above, it can be shown that

\[
  {\cal E}({\cal W}[I,U])= I U  \text{ and }  {\cal E}({\cal W}[I,U])= U^2/R_W
\]

\section{Proof of the main result}\label{sec:chara}

A network is \emph{transient} if for some vertex $v$ there is a $v$-flow of non-zero intensity with finite energy.
This definition is equivalent to the common one using random walks; see \cite{WoessBook2009}, Theorem 4.51.
Let $G/A/B$ denote the graph obtained from $G$ by contracting each of $A$ and
$B$ to a vertex.  
Our main result is:

\newtheorem*{thm1}{\autoref{char}}

\begin{thm1}
 A connected locally finite network $(G,r)$ is not in ${\cal O}_{HD}$ if and only if 
there are  transient vertex-disjoint subnetworks $A$ and $B$ such that the contraction $G/A/B$ admits a potential $\rho$ of finite energy with $\rho(A)\neq\rho(B)$.
\end{thm1}

\begin{proof}[Proof of the forward implication of \autoref{char}.]
Let $h$ be a non-constant harmonic function of finite energy.
As usual, $h$ induces a function $h_E$ on the directed edges via
$h_E((e,v,w)):=\frac{h(v)-h(w)}{r(e)}$. 
Since $h$ is non-constant, there is a directed edge $\vec{d}$ with
$h_E(\vec{d})>0$. 
Define $a:=init(\vec{d})$, $b:=ter(\vec{d})$ and $I:= h_E(\vec{d})$.
Let $A$ be the graph induced by vertices $v$ lying on some finite directed path
$\vec{W}$ from $v$ to $a$ with $h(\vec{e})>0$ for all $\vec{e}\in \vec{W}$, see
\autoref{cycle}.

   \begin{figure}[htpb]
\begin{center}
   	  \includegraphics[width=7.5cm]{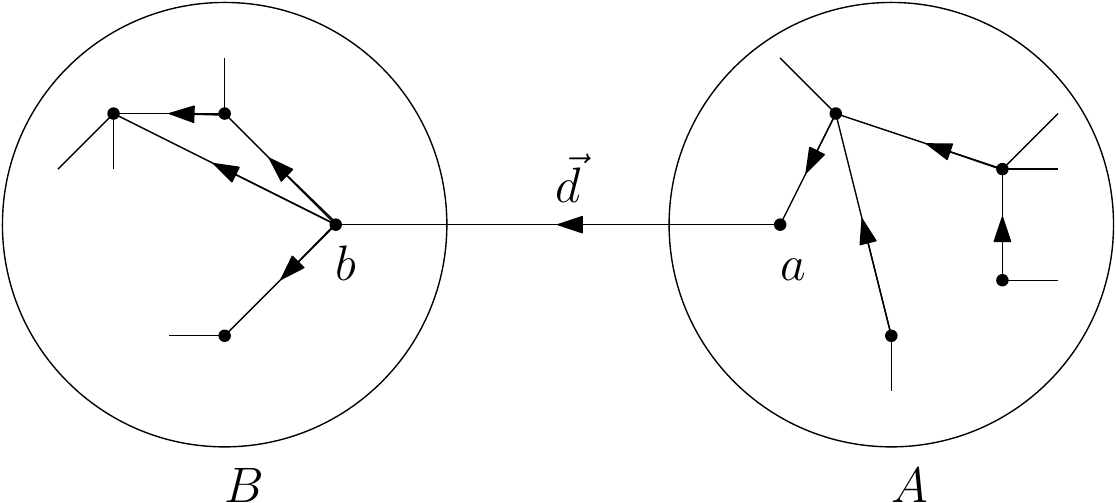}
   	  \caption{The construction of $A$. Only the edge-directions $\vec{e}$
with $h_E(\vec{e})> 0$ are drawn in this figure.}
   	  \label{cycle}
\end{center}
   \end{figure}

Our first task is to construct an $a$-flow of intensity $I$ in $A$ with finite energy,
beginning with the restriction $f'$ of $h_E$ to $A$, having accumulation at least $I$ at
$a$ and non-negative accumulation at every other vertex.
In order to obtain a function with accumulation exactly $I$ at $a$ and zero at
every other vertex, we apply \autoref{hilb} on the set of all antisymmetric
functions $g$ in $A$ with 

\begin{itemize}
\item $0\leq g(\vec{e})\leq h_E(\vec{e})$ if $h_E(\vec{e}) \geq 0$
 \item accumulation at least $I$ at $a$ and non-negative at every other vertex.
\end{itemize}

Note that $f'$ is in this set and the set is closed because all functions in
the set have energy at most ${\cal E}(h)$. \autoref{hilb} yields an element
$f$ with minimal energy. 
By minimality, $f$ has accumulation $I$ at $a$ and zero at every other
vertex. Thus $f$ witnesses that $A$ is transient.
Similarly the graph $B$ defined as the graph induced by the set of vertices $v$ lying on some finite directed path
$\vec{W}$ from $b$ to $v$ with $h(\vec{e})>0$ for all $\vec{e}\in \vec{W}$
is transient. The subnetworks $A$ and $B$ are disjoint because 
any vertex in $A\cap B$ is contained in a directed cycle $\vec{C}$ with
$h_E(\vec{c})>0$ for every $\vec{c}\in \vec{C}$, contradicting that $h_E$
satisfies (K2).

Having proved that $A$ and $B$ are transient, it remains to construct a potential $\rho$ of finite energy with $\rho(A)\neq \rho(B)$ in $G/A/B$.
Let $\bar{h}$ be the function obtained from $h$ by cutting off any values larger
than $h(a)$ and smaller than $h(b)$; more precisely, if $h(v)$ is bigger than
$h(a)$, we let $\bar{h}(v):=h(a)$ and 
if $h(v)$ is smaller than $h(b)$, we let $\bar{h}(v):=h(b)$. 
All other values are not changed.
 By the construction of $G/A/B$, the potential $\bar{h}$ is constant on every
contraction-set. So it defines a potential $\rho$ on $G/A/B$.
Since $\bar{h}$ has smaller energy than $h$ by construction, $\rho$ has finite
energy. 


\end{proof}

Before we can prove the converse direction, we need some intermediate results.

\begin{lem} \label{h_add2}
Let $G$ be a locally finite graph and let $r,r':E\rightarrow \mathbb{R}_{>0}$ be
resistance functions which differ only on finitely many edges. 
Then $(G,r)$ is in ${\cal O}_{HD}$ if and
only if $(G,r')$ is.
\end{lem}


\begin{proof}
By symmetry, it is sufficient to prove one direction.
It suffices to prove the assertion when $e=pq$ is the only edge with $r(e)\neq
r'(e)$, for applying this recursively, once for each edge $e$ with $r(e)\neq
r'(e)$, yields the general case.

Let $h$ be a non-constant harmonic function of finite energy in $(G,r)$. 
The desired harmonic function in $(G,r')$ will be constructed as a difference of
 two potentials of  $p$-$q$-flows.
The first one is $h$ considered as a potential in $(G,r')$.
The second is a multiple of the potential $f$ that induces the free current
${\cal F}[r',p,q,U=1]$: 
note that there is a real number $I$, depending linearly on $h(p)-h(q)$, such
that $h-I f$ is harmonic in $(G,r')$. Since $h$ and $I f$ have finite energy,
$h-If$ has finite energy, too.
As  ${\cal F}[r',p,q,U=1]={\cal F}[r,p,q,U=1]$, the difference $h-If$ is
non-constant  in $(G,r)$ and thus non-constant  in $(G,r')$, as well.

\end{proof}

With a similar proof one can strengthen the above Lemma, allowing $r$ and $r'$ 
to assume  the value zero and infinity. This has the same effect as contracting
and deleting edges.
In order to be able to do so, we need to impose the additional requirement that
the edges with infinite resistance  do not separate the graph, see \autoref{h_del}.
One can also state this stronger version of \autoref{h_add2} for non-elusive
harmonic functions. 
In that case the additional requirement is not needed if we consider harmonic
functions being non-constant in at least one connectedness-component.

After doing the calculation of the proof of the backward implication of
\autoref{char} in the following \autoref{cri_pot}, we will prove the  backward
implication of \autoref{char}.

\begin{prop}\label{cri_pot}
Let $\rho$ be a potential of finite energy in a connected locally finite network
$(G,r)$ with $\rho(p)-\rho(q)= U$  for some $p,q\in V$ and $U>0$.
Then for all $n\in \mathbb{N}$ and $I>0$ 
there exists a finite edge set $D$ and a resistance function $r_D$ with
$r_D|_{G-D}=r|_{G-D}$ such that ${\cal E}({\cal F}[r_D,p,q,I])\geq n$.
Moreover, we can choose $D$ disjoint from the set of edges $vw$ with
$\rho(v)=\rho(w)$.
\end{prop}

The idea of the proof of \autoref{cri_pot} is to make the resistances in $D$ so
large that the free effective resistance $R$ between $p$ and $q$ gets as large
as desired.
Thus ${\cal E}({\cal F}[r_D,p,q,I])=R I^2$ can be made as large as desired.

\begin{proof} Given $n,I$ and $U$, choose $\epsilon$ so small that
$\frac{U^2}{\epsilon} I^2\geq n$.
First of all, we define $D$ and $r_D$ so that $\rho$ has energy less than
$\epsilon$ in $(G,r_D)$.
Recall that the energy of $\rho$ in $(G,r)$ is
${\cal E}(\rho)=\sum_{vw\in E(G)} \frac{(\rho(v)-\rho(w))^2}{r(vw)}$.
Thus we can choose $D$ so large that the energy of $\rho$ in $(G-D,r|_{G-D})$ is
less than $\frac{\epsilon}{2}$. 
Note that we can choose $D$ disjoint from the set of edges $vw$ with
$\rho(v)=\rho(w)$.
As required, we set $r_D|_{G-D}=r|_{G-D}$.
To force the energy of $\rho$  to be less than $\epsilon$ in $(G, r_D)$, 
we choose $r_D$ on $D$ so large that the energy of $\rho$ in $(D,r_D|_{D})$ is
less than $\frac{\epsilon}{2}$.

Having defined $D$ and $r_D$, it remains to calculate the energy of the free
current ${\cal F}[r_D,I]$.
The definition of the free current yields
$
 {\cal E}({\cal F}[r_D,U])\leq \epsilon.
$
By \autoref{uri}, we obtain for the intensity of ${\cal F}[r_D,U]$ that
$
I({\cal F}[r_D,U])= \frac{{\cal E}({\cal F}[r_D,U])}{U} \leq
\frac{\epsilon}{U}. 
$
This yields:

\[
 {\cal E}({\cal F}[r_D,I])
= {\cal E}({\cal F}[r_D,U]) \frac{I^2}{I({\cal F}[r_D,U]) ^2}
=U\frac{I^2}{I({\cal F}[r_D,U]) } \geq 
\frac{U^2}{\epsilon} I^2\geq n.
\]

\end{proof}

We can now put the above tools together to prove the remaining part of
\autoref{char}.

\begin{proof}[Proof of the backward implication of \autoref{char}.] 
As $A$ and $B$ are transient, 
for some $a\in A$, $b\in B$, there are an $a$-flow $f_a$ of finite energy with intensity $I>0$
and a $b$-flow $f_b$ of finite energy with the same intensity $I$, which we extend both with the value zero to functions on $\vec{E}$.
Then $f:=f_b-f_a$ is an $a$-$b$-flow of intensity $I$ being zero on $E-E(A)-E(B)$.
Furthermore there is a potential $\rho$ of finite energy with $U:=\rho(A)-\rho(B)>0$ in $G/A/B$.

Finding a harmonic function in $(G,r)$ directly might be quite hard, 
instead 
we will manipulate the resistances using \autoref{cri_pot} such that we can find
a harmonic function in the manipulated network, and then we apply
\autoref{h_add2} to deduce that $(G,r)$ also admits a harmonic function.

In order to apply \autoref{cri_pot},
we extend $\rho$ to a potential $\rho'$ in $G$ by assigning the value of the
contraction set to all vertices in the set. Since $\rho$ has finite energy,
$\rho'$ does.
Thus \autoref{cri_pot}  yields for $(G,r)$, $\rho'$ and  $n> \ {\cal E}(f)$   a
set of edges $D$ and an assignment $r_D$ such that $ {\cal E}({\cal F}[r_D,I])
>{\cal E}(f)$.
Since $f$ is zero on $D$, $f$ is an $a$-$b$-flow in $(G,r_D)$. Therefore 
$ {\cal E}({\cal F}[r_D,I]) >{\cal E}(f)\geq {\cal E}({\cal W}[r_D,I])$. 
So ${\cal F}[r_D,I] - {\cal W}[r_D,I]$ is a non-constant harmonic function of
finite energy in $(G,r_D)$, giving rise to one in $(G,r)$ by \autoref{h_add2}.
\end{proof}

\section{Consequences of \autoref{char}}\label{sec:cor}

In this section we will derive further consequences from \autoref{char}.

\subsection{Networks not in ${\cal O}_{HD}$}

The following \autoref{r_fin} offers an example application of \autoref{char},
where the subnetworks $A$, $B$ and the potential $\rho$ can be constructed explicitly using the
properties of the graph. Its special case of unit resistances was already treated in \cite{SoardiBook}, Theorem 4.20.

\newtheorem*{cor1}{\autoref{r_fin}}

\begin{cor1}
 Let $(G,r)$ be a connected locally finite network.
If $G$ has a cut $F$ such that  $\sum_{e\in F} 1/r(e)$ is finite, and there are
two components of $G-F$ each containing a transient network,
then $(G,r)$ is not in ${\cal O}_{HD}$.
\end{cor1}

   \begin{figure}[htpb]
\begin{center}
   	  \includegraphics[width=5cm]{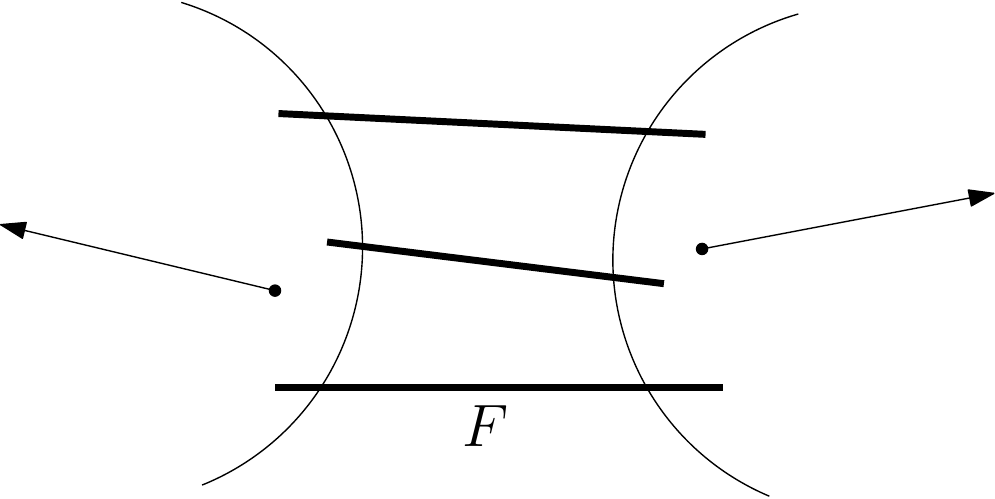}
   	  \caption{The situation of \autoref{r_fin}. The cut $F$, drawn thick,
separates the transient networks.}
   	  \end{center}
   \end{figure}

\begin{proof}
Pick for both $A$ and $B$ one of the above transient networks.
The potential $\rho$ is defined as follows: it assigns the value $1$ to every
vertex of the component of $G/A/B-F$ containing $A$, and zero to every other
vertex.
Recall that the energy of the potential $\rho$ is $\sum_{\{v,w\}\in E}
\frac{(\rho(v)-\rho(w))^2}{r(e)}$.
As $\sum_{e\in F} 1/r(e)$ is finite, $\rho$ has finite energy.
Thus \autoref{char} yields the assumption.
\end{proof}

\subsection{Networks in ${\cal O}_{HD}$}

In several occasions \autoref{char} can also be used in the other direction, to
prove that a network is in  ${\cal O}_{HD}$.
This is done in the following Corollaries \ref{w_barri} and \ref{cor_i}, which
we describe qualitatively at first. For simplicity, all edges have the
resistance $1$. 
Note that every infinite locally finite graph $G$ contains a sequence
$S_1,S_2,...$ of subgraphs such that $G-S_{n+1}$ has a finite component 
$C_i$ containing  $G[\bigcup_{i=1}^n S_i]$, see \autoref{fig:cors}.

   \begin{figure}[htpb]
\begin{center}
   	  \includegraphics[width=5cm]{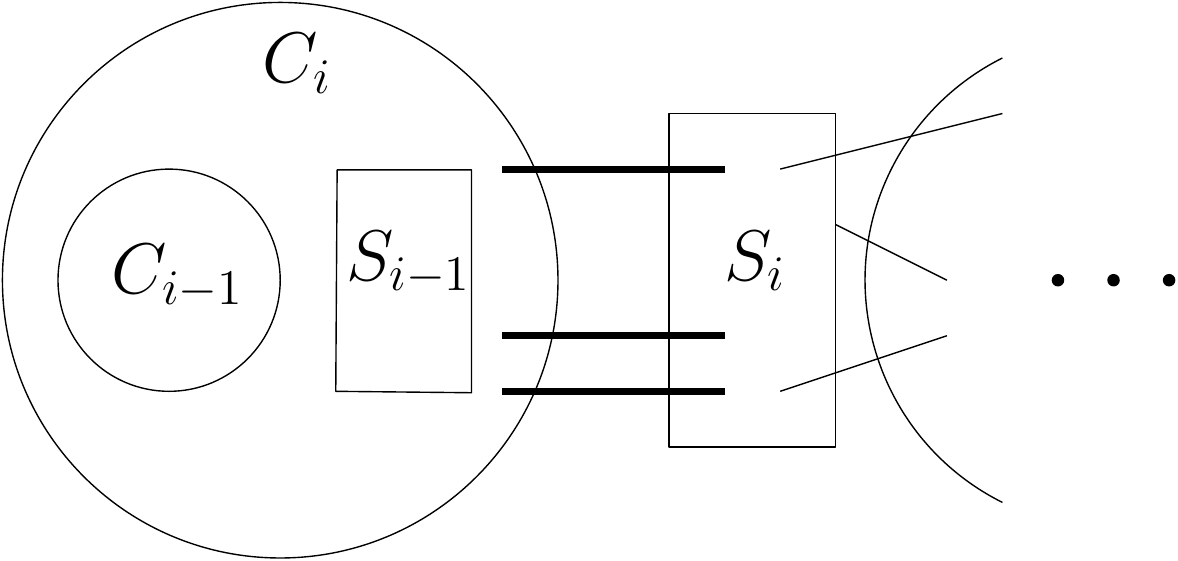}
   	  \caption{The separators $S_i$ and the finite components $C_i$. The
edges from $C_i$ to $S_i$ are drawn thick.
}
\label{fig:cors}
\end{center}
   \end{figure}

\autoref{cor_i} states that if there are only few edges from $C_i$ to $S_i$ for
sufficiently many $i$, then $G$ is in ${\cal O}_{HD}$.
In addition to that, \autoref{w_barri} states that if the graph-diameter of
$S_i$ is small for sufficiently many $i$, then $G$ is in ${\cal O}_{HD}$.

\sloppy
\begin{cor}[Thomassen \cite{thomassenCurrents2}] \label{w_barri}
 Let $(G,r)$ be a connected locally finite network with $r(e)=1$ for every edge
$e$.
Suppose $G$ contains infinitely many vertex-disjoint finite connected subgraphs $S_1,S_2,...$
such that $G-S_{n+1}$ has a finite component containing  $G[\bigcup_{i=1}^n
S_i]$.
If $\sum 1/diam(S_i)=\infty$, then $(G,r)$ is in ${\cal O}_{HD}$.
\end{cor}
\fussy

Here $diam(S_i)$ is the graph-diameter of $S_i$. Lyons and Peres
\cite{LyonsBook} proved a generalization of \autoref{w_barri} to arbitrary resistances which is proved by \autoref{char} similarly.

\begin{proof}
Assume there is a non-constant harmonic function of finite energy in $G$: 
\autoref{char} yields transient vertex-disjoint subnetworks $A$ and $B$ and a potential $\rho$ of finite energy with $\rho(A)\neq \rho(B)$.
By extending the value of the contraction set to all vertices of the set, $\rho$
defines a potential $\rho'$ on $G$, having finite energy.

Our aim is to show that $\rho'$ has infinite energy, which yields the desired
contradiction.
For this, it will be useful to find vertices $a_i\in A\cap S_i$ and $b_i\in B\cap S_i$ for all but
finitely many $i$.

Let us start finding these vertices. 
Let $A'$ be an infinite connected component of $A$ and pick $a\in A'$.
Let $n_a$ be the distance in $G$ between $a$ and $S_1$. 
As $G-S_{n+1}$ has a finite component containing  $G[\bigcup_{i=1}^n S_i]$ and
the subgraphs $S_i$ are disjoint, 
it follows for all $j\geq n_a$ that the vertex $a$ is contained in the finite
component of $G-S_{j+1}$ containing $S_1$. 
Thus the connected infinite set $A'$ contains a vertex $a_{j+1}$ of the separator
$S_{j+1}$. We define $B', n_b$ and $b_{j+1}$ analogously for $b$ instead of $a$.

Define $U:= \rho'(A)-\rho'(B)$ and ${\cal E}(\rho'|S_i):=\sum_{\{v,w\}\in S_i}
\frac{(\rho'(v)-\rho'(w))^2}{r(vw)}$.
Having proved for all $i\geq m:=max\{n_p,n_q\}+1$ that there are $a_i\in A\cap
S_i$ and $b_i\in B\cap S_i$, we calculate:

\[
 {\cal E}(\rho') \geq \sum_{i\geq m} {\cal E}(\rho'|S_i)\geq  \sum_{i\geq m}
{\cal E}({\cal F}[S_i, a_i,b_i, U])\geq U^2 \sum_{i\geq m} 1/diam(S_i)=\infty
\]

as desired.
\end{proof}

For the next corollary, we need the following definition:
Given a subgraph $C_i$ of $G$, we let ${\bf RN}(C_i)$ denote the
\emph{resistance neighborhood of $C_i$}, which is defined as $\sum
\frac{1}{r(e)}$, summing over all edges $e$ having one end-vertex in $C_i$ and
one outside.
In the case where all resistances are $1$, the number ${\bf RN}(C_i)$ is the
size of the neighborhood of $C_i$.
If a network is not in  ${\cal O}_{HD}$, then by \autoref{char} it contains a transient subnetwork, witnessing that the network itself is transient.
Thus the Nash-Williams-criterion \cite{LyonsBook} for not transient graphs yields:

\begin{cor}\label{cor_i}
 Let $(G,r)$ be a connected locally finite network.
Suppose $G$ contains infinitely many vertex-disjoint finite connected subgraphs $S_1,S_2,...$
such that $G-S_{n+1}$ has a finite component $C_i$ containing 
$G[\bigcup_{i=1}^n S_i]$.
If $\sum \frac{1}{{\bf RN}(C_i)}=\infty$, then $(G,r)$ is in ${\cal O}_{HD}$.
\end{cor}

The special case of unit resistances was treated by Thomassen in \cite{thomassenCurrents}.

\subsection{${\cal O}_{HD}$ and the deletion of edges}

The following result extends the well-known fact \cite{agelos_unique_flow} that a network $(G,r)$ with
$\sum_{e\in E(G)} 1/r(e)<\infty$ is in  ${\cal O}_{HD}$.
With a light abuse of notation, let $G-S$ denote the graph obtained from $G$ by deleting the set of edges $S$ and then all isolated vertices.

\newtheorem*{cor3}{\autoref{G-S}}
\begin{cor3}
Let $(G,r)$ be a connected locally finite network, and let $S$ be a set of edges 
such that $G-S$ is connected and $\sum_{e\in S} 1/r(e)$ is finite.
The network $(G-S,r)$ is in ${\cal O}_{HD}$ if and only if $(G,r)$ is.
\end{cor3}

The condition that $\sum_{e\in S} 1/r(e)$ is finite is best possible in the
following strong sense.
Given any set $S$ with $\sum_{e\in S} 1/r(e)=\infty$, there is a network
$N_1=(G,r)$ that is in ${\cal O}_{HD}$ but $(G-S,r)$ is not.
The converse is also true: 
given any set $S$ with $\sum_{e\in S} 1/r(e)=\infty$, 
there is a network
$N_2=(G,r)$ that is not in ${\cal O}_{HD}$ but $(G-S,r)$ is.
In particular, the best possible terms for both directions of the upper theorem
agree.

In the following, we construct $N_1$ and $N_2$, starting with $N_1$.
Letting $(G-S,r)$ be a double ray of which the resistances sum up to $1$, ensures by \autoref{r_fin} that $(G-S,r)$ is not in ${\cal O}_{HD}$.
We attach the edges of $S$ to the double ray so that the graph $G$ is an
infinite ladder 
and every edge of $S$ is a rung of that ladder, see \autoref{fig:bsp1}.
With \autoref{satz}, proved in Section \ref{sec:geo},
it is straightforward to check that $G$ is in ${\cal O}_{HD}$.

   \begin{figure}[htpb]
\begin{center}
   	  \includegraphics[width=5cm]{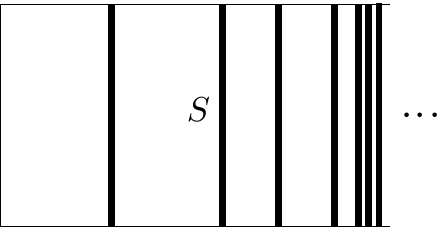}
   	  \caption{The network $(G,r)$ where the set $S$ is thick.}
   	  \label{fig:bsp1}
\end{center}
   \end{figure}

Having constructed $N_1$, we now construct $N_2$.
Letting $G-S$ be the infinite ladder and choosing the resistances so that
$\sum_{e\in E} 1/r(e)=1$, 
ensures $(G-S,r)$ is in ${\cal O}_{HD}$ by
\autoref{G-S} or \autoref{cor_i}.

Thus it remains to attach the set $S$ so that $(G,r)$ is not in ${\cal O}_{HD}$, which is
done as follows:
as $\sum_{e\in S} 1/r(e)=\infty$, we can partition $S$ into finite sets $H_{i}$,
where $i\in \mathbb{N}$, so that $\sum_{e\in H_i} 1/r(e)\geq 2^{i}$.
Let $e_0,e_1,...$ be any enumeration of the horizontal edges of the ladder.
For every edge $e_i$, we attach each edge of $H_i$ between the end-vertices of
$e_i$, see \autoref{fig:bsp2}.
This has the same effect as assigning a resistance smaller than $2^{-i}$ to the
edge $e_i$.
Thus by \autoref{r_fin} the network $(G,r)$ is not in ${\cal O}_{HD}$.

   \begin{figure}[htpb]
\begin{center}
   	  \includegraphics[width=5cm]{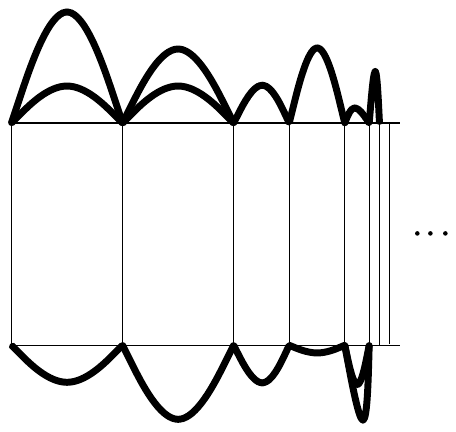}
   	  \caption{The set $S$, drawn thick, attached to the infinite ladder.}
   	  \label{fig:bsp2}
\end{center}
   \end{figure}


Having seen that \autoref{G-S} is best possible, we proceed with its proof.

\begin{proof}[Proof of the forward implication of \autoref{G-S}.] 
Our aim is to find transient vertex-disjoint subnetworks $A$ and $B$ and a potential $\rho'$ of finite energy with $\rho'(A)\neq \rho'(B)$ in $G/A/B$ to apply \autoref{char} in $G$.
Applying \autoref{char} in $G-S$ yields the desired $A$ and $B$ and a potential $\rho$ of finite energy with $\rho(A)< \rho(B)$ in $G/A/B-S$.
Define the potential $\rho'$ via:

\[
\rho'(v):=
\begin{cases}
    \rho(v) \text{ if } \rho(A)\leq \rho(v) \leq \rho(B) \\
    \rho(A) \text{ if } \rho(v) \leq \rho(A) \\
    \rho(B) \text{ if } \rho(B)\leq \rho(v) \\
    \rho(A) \text{ if } v\notin G-S
  \end{cases} 
\]

As $\rho'(A)\neq \rho'(B)$, it remains to check that $\rho'$ has finite energy:
its energy is at most that of $\rho$ plus the energy on the edges of $S$ which is at most
$P^2 \sum_{e\in S} 1/r(e)$, where $P:=|\rho'(A)-\rho'(B)|$.
This completes the proof.
\end{proof}

Before we can prove the converse direction, we need some intermediate results.
The following \autoref{h_del} is \autoref{G-S} specialized to the case that $S$
is finite and is proved similar to \autoref{h_add2}.

\begin{lem} \label{h_del}
Let $(G,r)$ be a connected locally finite network and let $S$ be a finite
set of edges such that $G-S$ is connected.
Then $(G-S,r)$ is in ${\cal O}_{HD}$ if and only if $(G-S,r)$ is.
\end{lem}

Recall that an \emph{$a$-flow of intensity $I$} is an antisymmetric function having accumulation $I$ at $a$ and satisfying $(K1)$ at every other vertex.
Intuitively, the following \autoref{small_flow} states that if 
an $a$-flow of finite energy has small enough values on a set of edges $S$, then the $a$-flow gives rise to an $a$-flow of finite energy in $G-S$.

\begin{prop}\label{small_flow}
Let $f_a$ be an $a$-flow of intensity $I$ with finite energy in a connected locally finite network $(G,r)$ and
let $S$ be a set of edges such that $\sum_{s\in S} |f_a(\vec{s})|\leq
I/4$.
Then there is an $a$-flow $f_a'$ in $(G-S,r)$ with intensity at least $I/2$ satisfying
 $0\leq f_a'(\vec{e})\leq f_a(\vec{e})$ if $f_a(\vec{e}) \geq 0$.
In particular, $f_a'$ has finite energy.
\end{prop}

\begin{proof}
In order to obtain $f_a'$, we apply \autoref{hilb} on the set of all
antisymmetric functions $g$ in $G-S$ with 

\begin{itemize}
\item $0\leq g(\vec{e})\leq f_a(\vec{e})$ if $f_a(\vec{e}) \geq 0$,
\item accumulation at least $I$ at $a$,
\item $\sum_{v\in V-\{a\}} |accu(v)|\leq I/2$, where $accu(v)$ is the accumulation of $g$ at
$v$.
\end{itemize}

Note that the restriction of $f_a$ to $G-S$ is in this set and the set is closed
because all functions in the set have energy at most ${\cal E}(f_a)$. 
\autoref{hilb} yields an element $f^*$ with minimal energy. 
By minimality, $f^*$ satisfies (K1) at every vertex that is not $a$ or a
neighbor of $a$.
Let $f_a'$ be the function obtained from $f^*$ by changing the values of the
edges between $a$ and its neighbors such that (K1) is satisfied at all
neighbors of $a$.
By minimality we can assume that $0\leq f_a'(\vec{e})\leq f^*(\vec{e}) \leq
f_a(\vec{e})$ if $f_a(\vec{e}) \geq 0$. 
As we demanded $\sum_{v\in V-\{a\}} |accu(v)|\leq I/2$ for $f^*$, the accumulation of $f_a'$
at $a$ is at least $I/2$. 
This completes the proof.
\end{proof}

\begin{proof}[Proof of the backward implication of \autoref{G-S}.] 
Applying \autoref{char} to $(G,r)$ yields vertex-disjoint subnetworks $A$ and $B$, an $a$-flow $f_a$ of intensity $I>0$ with finite energy in $A$,
a $b$-flow $f_b$ of intensity $I>0$ with finite energy in $B$ and a potential $\rho$ of finite energy with $\rho(A)\neq \rho(B)$ in $G/A/B$.
Let us first consider the special case where $\sum_{s\in S} |f_a(\vec{s})|+\sum_{s\in S} |f_b(\vec{s})|<\epsilon$.
Since by \autoref{small_flow} the functions $f_a$ and $f_b$ give rise to an $a$-flow of non-zero intensity with finite energy in $A- S$ and a $b$-flow of 
non-zero intensity with finite energy in $B- S$,
it suffices to find a potential $\rho'$ of finite energy with $\rho'(A-S)\neq \rho'(B-S)$ in $(G-S)/ (A-S)/(B-S)$ for proving the special case applying once again \autoref{char}.
Since $G/A/B- S$ is obtained from $(G-S)/ (A-S)/(B-S)$ by identifying vertices, let $\rho'$ of a vertex in $(G-S)/ (A-S)/(B-S)$
be the $\rho$-value of the corresponding identification-set.
As $\rho$ has finite energy and $\rho(p)\neq \rho(q)$, the potential $\rho'$ has
finite energy and $\rho'(A-S)\neq \rho'(B-S)$, proving the special case by \autoref{char}.

Having treated the special case where $\sum_{s\in S} |f_a(\vec{s})|+\sum_{s\in S} |f_b(\vec{s})|<\epsilon$, it
remains to deduce the general case from this special case.
For this purpose, we first show that $\sum_{s\in S} |f_a(\vec{s})|$ is finite.
Applying  Cauchy-Schwartz-inequality $(\sum_{s\in S} x_s y_s)^2\leq
\sum_{s\in S} x_s^2 \sum_{s\in S} y_s^2$ with $x_s:=1/\sqrt{r(s)}, y_s:=\sqrt{r(s)}|f_a(\vec{s})|$, yields:

\[
\left(\sum_{s\in S} |f_a(\vec{s})|\right)^2\leq \sum_{s\in S} \frac{1}{r(s)}
\sum_{s\in S} r(s) f_a^2(s) 
\]

As both terms on the right side are finite, $\sum_{s\in S} |f_a(\vec{s})|$ is
finite.
Thus we can partition $S$ into $S_1$ and $S_2$ such that 
$\sum_{s\in S_1} |f_a(\vec{s})|+\sum_{s\in S} |f_b(\vec{s})|<\epsilon$ 
and $S_2$ is finite.
By the special case, we obtain that $G-S_1$ is not in ${\cal O}_{HD}$.
Hence by \autoref{h_del} $G-S_1-S_2$ is not in ${\cal O}_{HD}$, completing the proof.
\end{proof}

\section{Non-elusive harmonic functions}\label{sec:geo}

Recently, Georgakopoulos \cite{ agelos_unique_flow} introduced the concept of
non-elusiveness, which we will present now. 
One can define the accumulation of $\varphi$ at a finite cut as well:

\[
 \varphi(X,X'):= \sum_{\vec{e}|init(\vec{e})\in X, ter(\vec{e})\in X'}
\varphi(\vec{e})
\]

A $p$-$q$-flow with intensity $I$ is called \emph{non-elusive} if for every
finite cut $(X,X')$ with $p$ and $q$ on the same, the accumulation is zero. 
It follows for $p\in X, q\in X'$ that $\varphi(X,X')=\varphi(\{p\},V-
\{q\})=
\varphi(V- \{p\},\{q\})=I$.

Note that in a finite network every flow is non-elusive.
In some sense, \emph{non-elusiveness} ensures that (K1) also holds for the ends
of the Freudenthal-compactification.
For details see \cite{ agelos_unique_flow}.

A \emph{harmonic function is non-elusive} if the induced antisymmetric function
is non-elusive.
Notice that there is a non-constant non-elusive harmonic function (of finite
energy) in a connected graph if and only if there is one in at least one maximal
$2$-connected subgraph.
In particular, non-elusive harmonic functions on trees are constant.

In this section we will generalize \autoref{w_barri} to extend a theorem of
Georgakopoulos about non-elusive harmonic functions.
For this, we need some definitions.
A subgraph $S$ of a graph $G$ is called a \emph{barricade around the edge $e\in
E(G-S)$} if both of the following requirements hold, see \autoref{fig:barri}:

\begin{enumerate}
\item The component of $G-S$ containing $e$ is finite and called the
\emph{barricaded area $A(S,e)$}.    \label{2}
\item The intersection of $S$ with any component of $G-A(S,e)$ is connected.
\label{3}
\end{enumerate}

   \begin{figure}[htpb]
\begin{center}
   	  \includegraphics[width=5cm]{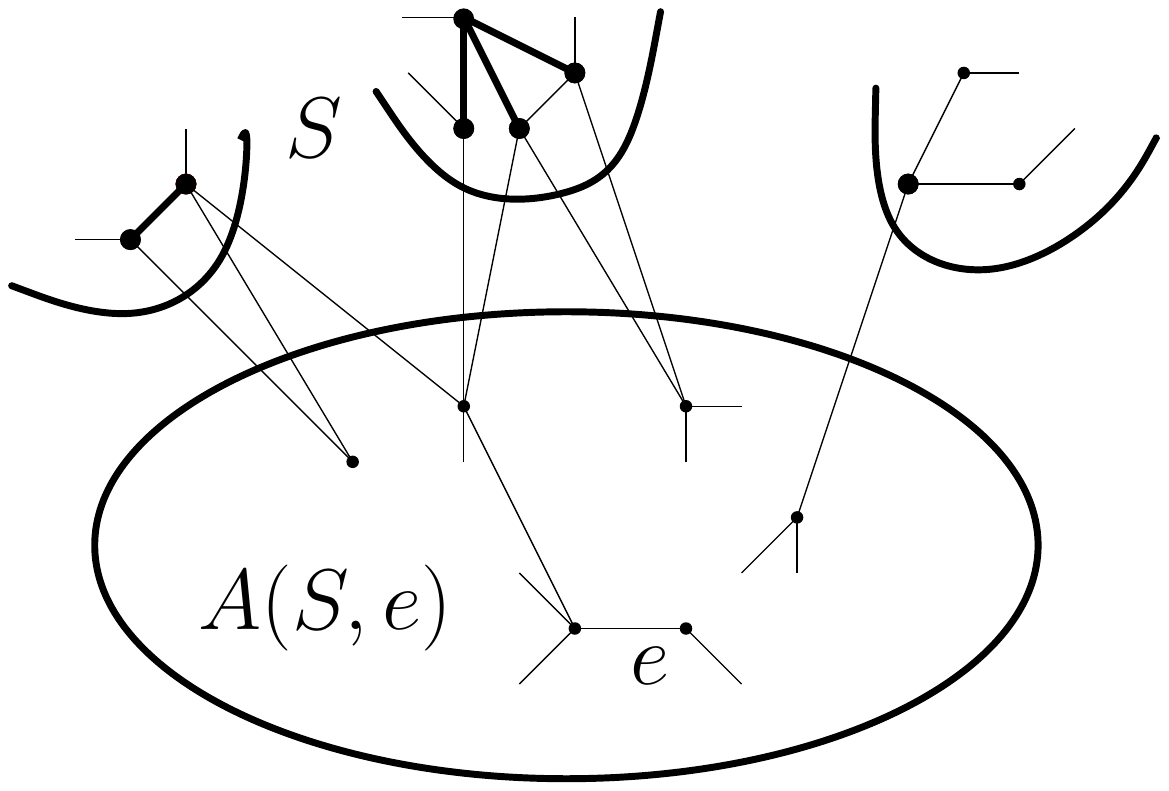}
   	  \caption{An example of a barricade. No proper subgraph of the
barricade, drawn thick, is again a barricade. Deleting a vertex, violates
requirement \ref{2}.
Deleting an edge, violates requirement \ref{3}.}
   	  \label{fig:barri}
\end{center}
   \end{figure}

The \emph{boundary $\partial S$ of a barricade $S$} is the neighborhood of the
barricaded area $A(S,e)$.
For a subset $C$ of a barricade $S$, define $\partial C:=\partial S\cap C$.
Let $R(x \leftrightarrow y ; G,r)$, or just $R(x \leftrightarrow y ; G)$ if $r$
is fixed, denote the effective resistance 
between the vertices $x$ and $y$ in a connected finite network $(G,r)$. 

For a component $C$ of a barricade, define the \emph{weak effective resistance
diameter} ${\bf wRD}$ by:

\[
 {\bf wRD}(C) := sup\{R(x \leftrightarrow y ; C) | x, y \in \partial C\}
\]

Furthermore, define the \emph{weak effective resistance diameter} ${\bf wRD}$ of
a barricade as the sum of the weak effective resistance diameters of the
components of the barricade.
Note that in the case of unit resistances, ${\bf wRD}(S)$ is at most the graph
diameter of $S$.

The following theorem states that if the weak effective resistance diameters of
a sequence of barricades does not grow too fast, 
then every non-elusive  harmonic function of finite energy is constant.

\begin{thm} \label{satz}
Let $(G,r)$ be a connected locally finite network which has for every edge $e\in
E(G)$ infinitely many edge-disjoint barricades  $S_1,S_2,...$ around $e$
with $\sum_{n=1}^\infty 1/{\bf wRD}(S_n) =\infty$.
Then every non-elusive harmonic function of finite energy is constant.
\end{thm}

\autoref{satz} generalizes the \emph{Unique Currents from Internal
Connectivity}-Theorem from Lyons and Peres \cite{LyonsBook} which implies
\autoref{w_barri}.
As \autoref{satz} can be meaningfully applied to graphs with more than one end,
\autoref{satz} is stronger than the aforementioned Theorem, which only holds for
one ended graphs.
If $G$ is $2$-connected, then in \autoref{satz} it is enough to check the
condition just for one edge $e$:

\begin{cor}
 Let $(G,r)$ be a $2$-connected locally finite network which has, for some edge
$e\in E(G)$, infinitely many edge-disjoint barricades  $S_1,S_2,...$ around $e$
with $\sum_{n=1}^\infty 1/{\bf wRD}(S_n) =\infty$.
Then every non-elusive harmonic function of finite energy is constant.
\end{cor}

\begin{proof}
Given infinitely many edge-disjoint barricades  $S_1,S_2,...$ around $e$ with
$\sum_{n=1}^\infty 1/{\bf wRD}(S_n) =\infty$, 
we will show for every edge $e'$ that all but finitely many of these barricades
are barricades around $e'$, too.
Let $\cal S$ be any set of edge-disjoint barricades separating $e$ and $e'$. 
It is sufficient to prove that $\cal S$ is finite.
Let $P$ be any finite path containing $e$ and $e'$. 
It suffices to show that each vertex $p$ on $P$ is contained in only finitely
many barricades of $\cal S$.
The $2$-connectedness of $G$ yields that if the vertex $p$ is in some $S\in \cal
S$, then, by requirement \ref{3} of the barricade-properties, $S$ contains at
least one edge incident with $p$, as well. 
As $G$ is locally finite, $\cal S$ is finite, completing the proof.
\end{proof}

The following theorem of Georgakopoulos can be deduced by \autoref{satz}.

     \begin{thm}[Georgakopoulos \cite{agelos_unique_flow}]  \label{u_geo}
Let $(G,r)$ be a connected locally finite network such that $\sum_{e\in E}
r(e)<\infty$.
Then every non-elusive harmonic function of finite energy is constant.
      \end{thm}

\begin{proof}

For the proof, we first check the following fact.

\begin{equation}\label{G_hat_barri}
\begin{minipage}[c]{0.8\textwidth}
In every locally finite graph for every edge $e$ there are infinitely many
disjoint finite barricades $S_n$ around $e$.
\end{minipage}
\end{equation}

Assume finitely many finite barricades around $e$ are already constructed, our
task is to define one more being disjoint with the previous ones.
As $G$ is locally finite, there is a finite connected subgraph $A$ containing
$e$ and all so far constructed barricades.
Since $G$ is locally finite, there is a finite barricade with $A$ as barricaded
area, proving~(\ref{G_hat_barri}).

By (\ref{G_hat_barri}) for every edge $e$ there are infinitely many
edge-disjoint barricades $S_n$ around $e$.
Define $D:=\sum_{e\in E}r(e)$. As $\sum_{n=1}^\infty 1/{\bf wRD}(S_n) \geq
\sum_{n=1}^\infty 1/D =\infty$, \autoref{satz} yields the assertion.
\end{proof}

\section{Proof of \autoref{satz}}\label{sec:barri_proof}

As later on in the proof of \autoref{satz}, we assume there exists a non-elusive
non-constant harmonic function $h$ of finite energy in $(G,r)$.
Before proving \autoref{satz}, we will show (\ref{u_eps}), transforming the
resistance condition $\sum_{n=1}^\infty 1/{\bf wRD}(S_n) =\infty$ into a voltage
condition.
Later on, we will use the voltage condition instead of the resistance condition.
Define the \emph{voltage at a barricade $S_i$} as $U(h|S_i):=\sum_{A}
max\{|h(a_{1})-h(a_{2})|  \text{, where }  a_{1},a_{2}\in \partial A \}$,
summing over all components $A$ of $S_i$.

\begin{equation}\label{u_eps}
\begin{minipage}[c]{0.8\textwidth}
For every $\epsilon>0$ there is a barricade $S_i$ with voltage 
$U(h|S_i)<\epsilon$.
\end{minipage}
\end{equation}

Intuitively, this means that small resistances at the barricades imply small
voltages at the barricades.

\emph{Proof of (\ref{u_eps}).}
The tools of Section \ref{sec:def} hold only for connected network.
As $S_i$ is not necessarily connected, 
we will construct a connected auxiliary graph $S_i'$ by identifying vertices of
different components of $S_i$ for applying the tools in $S_i'$.

To begin with the construction of $S_i'$, we enumerate the components of $S_i$
with $1,.., k$.
For every component $A_j$, in the boundary $\partial A_j$ we have vertices $s_j$
and $t_j$ for which $|h(s_j)-h(t_j)|$ attains its maximum.

We obtain the auxiliary graph $S_i'$ from $S_i$ by identifying $t_j$ with
$s_{j+1}$ for all $j\leq k-1$.
Note that the effective resistance between $s_{1}$ and $t_{k}$ in $S_i'$ is at
most ${\bf wRD}(S_i)$.
Let ${\cal F}$ be the free current in $S_i'$ between $s_{1}$ and $t_{k}$ with
voltage $U(h|S_i)$.

As $h$ induces a potential in $S_i'$, in $S_i'$ we can relate ${\bf wRD}(S_i)$
to $U(h|S_i)$ in the following way:

\[
U(h|S_i)^2  =  (U\left(  {\cal F}\right ))^2\leq ^{\autoref{uri}}
\]
\[
\leq {\bf wRD}(S_i) \cdot {\cal E}({\cal F}) \leq^{\text{minimizing property of
$\cal F$}} {\bf wRD}(S_i)\cdot  {\cal E}(h|_{S_i})
\]
Here ${\cal E}(h|_{S_i})$ is the energy of $h$ on the edges of $S_i$.
If we assume in contrast to (\ref{u_eps}) that there is an $\epsilon>0$ such
that $U(h|S_i)\geq \epsilon$ for all $i$, then we get a contradiction to the
fact that energy is finite as follows:

\[
 {\cal E}(h)\geq \sum_{i}  {\cal E}(h|_{S_i}) \geq^{\text{last inequation}}
\sum_{i}\frac{U(h|S_i)^2}{{\bf wRD}(S_i)}\geq \epsilon^2  \sum_{i}\frac{1}{{\bf
wRD}(S_i)}=\infty
\]

This proves (\ref{u_eps}). We can now prove \autoref{satz}.

\begin{proof}[Proof of \autoref{satz}]

Assume there exists a non-elusive non-constant harmonic function $h$ of finite
energy in $(G,r)$.
As usual, $h$ induces a function $h_E$ on the directed edges via
$h_E((e,v,w)):=\frac{h(v)-h(w)}{r(e)}$. 
Since $h$ is non-constant, there is a directed edge $\vec{e}$ with
$h_E(\vec{e})>0$.
The voltage condition (\ref{u_eps}) yields a barricade $S_i$ around $e$ with
$U(h|S_i)<\epsilon$ for $\epsilon:= h_E(\vec{e})r(e)$.

To obtain a contradiction, we seek a cycle violating Kirchhoff's cycle law.
This will be done in two steps. Firstly, we find a cycle heavily violating
Kirchhoff's cycle law in an auxiliary graph $G'$ 
which we obtain from $G$ by contracting each component of $G-A(S_i,e)$ to a
vertex.
Secondly, we extend this cycle to a cycle in $G$ using only edges of $S_i$. As
$U(h|S_i)<\epsilon$, we will be able to show that in this new cycle (K2) is
still violated.

Let us now construct the above mentioned cycle in $G'$. As the barricaded area
$A(S_i,e)$ is finite and therefore $G-A(S_i,e)$ has only finitely many
components, $G'$ is finite.
Since $h_E$ is non-elusive, the restriction $h'_E$ of $h_E$ to $E(G')$ is a flow
of intensity zero in $G'$.
Thus \autoref{kreis} yields a directed cycle $\vec{C'}$ in $G'$ with $\vec{e}\in
\vec{C'}$ and $h'_E(\vec{c})>0$ for every $\vec{c}\in\vec{C'}$.

   \begin{figure} [htpb]   
\begin{center}
   	  \includegraphics[width=7cm]{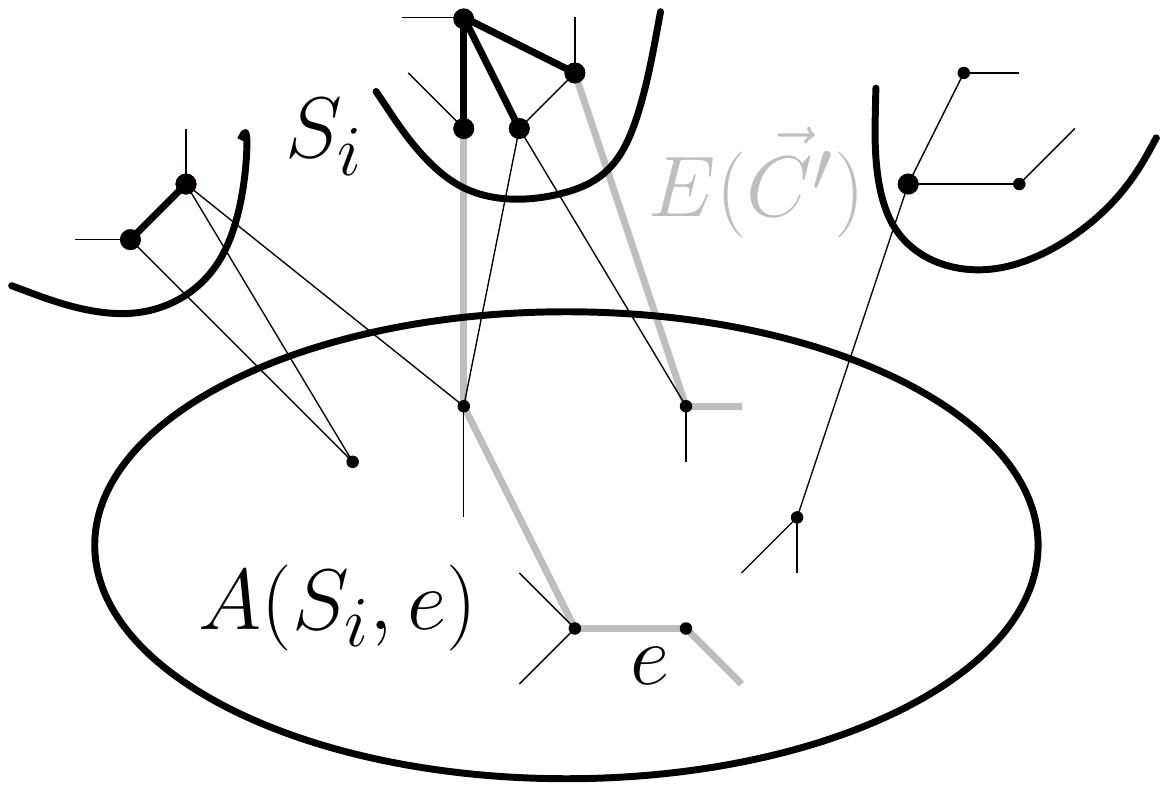}
   	  \caption{The construction of $C$. The gray set $E(\vec{C'})$ can be
extended to a cycle in $G$ by just adding edges of the barricade $S_i$ drawn
thick in this figure.}
   	  \label{fig:barri_cycle}
\end{center}
   \end{figure}

Having found this cycle $\vec{C'}$ in $G'$,  we will extend its edge set
$E(\vec{C'})$, considered as a set of edges in $G$, into a cycle in $G$; see
\autoref{fig:barri_cycle}.
Note that $E(\vec{C'})$ has at most two vertices in any component of
$G-A(S_i,e)$.
Let $K$ be any component of $G-A(S_i,e)$ where $E(\vec{C'})$ has  exactly two
vertices, say $v$ and $w$. 
Since $S_i$ is a barricade, $v$ and $w$ are contained in $S_i$ and thus there is
a $v$-$w$-path $W_K$ in $S_i \cap K$.
The desired cycle $C$ in $G$ is the union of $E(\vec{C'})$ with such paths $W_K$
for all $K$.
Indeed, as different paths $W$ are disjoint and intersect $\vec{C'}$ only in
end-vertices, this union is in fact a cycle.


For the desired contradiction, it remains to check that $\vec{C}$ violates
Kirchhoff's cycle law:
the voltage-sum of the directed edges in $\vec{C'}$ is at least $\epsilon=
h_E(\vec{e}) r(e)$, whereas the sum over the voltages of the edges of $S_i$ is
at most $U(h|S_i)<\epsilon$.
Thus $h_E$ violates (K2), completing the proof.

\end{proof}

\section{Acknowledgements}

I am very grateful to Agelos Georgakopoulos for his great supervision of this project.

\bibliographystyle{plain}
\bibliography{literatur1.bib}

\end{document}